\documentclass[letterpaper,11pt]{amsart}
\usepackage{amsmath,amsthm,amsfonts,amssymb,graphicx,mathtools}
\usepackage[usenames,dvipsnames,svgnames,table]{xcolor}
\usepackage[foot]{amsaddr}
\usepackage[pdftitle={Better Bounds for Boxicity and Poset Dimension},colorlinks=true,citecolor=black,linkcolor=black,urlcolor=blue!80!black]{hyperref}
\usepackage[noabbrev,capitalise]{cleveref}
\usepackage[longnamesfirst,numbers,sort&compress]{natbib}

\makeatletter
\def\NAT@spacechar{~}
\makeatother
\theoremstyle{plain}
\newtheorem{thm}{Theorem}
\newtheorem{lem}[thm]{Lemma}
\newtheorem{prop}[thm]{Proposition}
\newtheorem{cor}[thm]{Corollary}
\newtheorem*{claim}{Claim}
\crefname{lem}{Lemma}{Lemmas}
\crefname{prop}{Proposition}{Propositions}
\crefname{thm}{Theorem}{Theorems}
\crefname{cor}{Corollary}{Corollaries}
\newcommand{\arXiv}[1]{arXiv:\,\href{http://arxiv.org/abs/#1}{#1}}
\newcommand{\msn}[1]{MR:\,\href{http://www.ams.org/mathscinet-getitem?mr=MR#1}{#1}}
\newcommand{\doi}[1]{doi:\,\href{http://dx.doi.org/#1}{#1}}
\DeclarePairedDelimiter\ceil\lceil\rceil
\DeclarePairedDelimiter\floor\lfloor\rfloor
\DeclareMathOperator{\bx}{box}

\DeclareMathOperator{\dist}{dist}
\DeclareMathOperator{\ltw}{ltw}
\DeclareMathOperator{\tw}{tw}
\renewcommand{\ge}{\geqslant}
\renewcommand{\le}{\leqslant}
\renewcommand{\geq}{\geqslant}
\renewcommand{\leq}{\leqslant}

\begin{document}

\title[Better Bounds for Poset Dimension and Boxicity]{Better Bounds for Poset Dimension\\ and Boxicity}

\author{Alex Scott}
\address[A. Scott]{Mathematical Institute, University of Oxford, Oxford OX2 6GG, United Kingdom}
\email{scott@maths.ox.ac.uk}
\thanks{Alex Scott is supported by a Leverhulme Trust Research Fellowship.}

\author{David R. Wood}
\address[D. R. Wood]{School of Mathematics, Monash University, Melbourne, Australia}
\email{david.wood@monash.edu}
\thanks{David Wood is supported by the Australian Research Council.}


\date{17th June 2018}

\maketitle

\begin{abstract} 
The \emph{dimension} of a poset $P$ is the minimum number of total orders whose intersection is $P$.  We prove that the dimension of every poset whose comparability graph has maximum degree $\Delta$ is at most $\Delta\log^{1+o(1)} \Delta$. This result improves on a 30-year old bound of F\"uredi and Kahn, and is within a $\log^{o(1)}\Delta$ factor of optimal. We prove this result via the notion of boxicity. The \emph{boxicity} of a graph $G$ is the minimum integer $d$ such that $G$ is the intersection graph of $d$-dimensional axis-aligned boxes. We prove that every graph with maximum degree $\Delta$ has boxicity at most $\Delta\log^{1+o(1)} \Delta$, which is also within a $\log^{o(1)}\Delta$ factor of optimal. We also show that the maximum boxicity of graphs with Euler genus $g$ is $\Theta(\sqrt{g \log g})$, which solves an open problem of Esperet and Joret and is tight up to a constant factor. 
\end{abstract}

%%%%%%%%%%%%%%%%%%%%%%
\section{Introduction}
\label{Introduction}

\subsection{Poset Dimension and Degree}

The \emph{dimension} of a poset $P$, denoted by $\dim(P)$, is the minimum number of total orders whose intersection is $P$. Let $\dim(\Delta)$ be the maximum dimension of a poset whose comparability graph has maximum degree at most $\Delta$. Several bounds on $\dim(\Delta)$ have been proved in the literature. In unpublished work, R\"odl and Trotter proved the first upper bound, $\dim(\Delta)\leq 2\Delta^2+2$; see \citep[pp.~165--166]{Trotter92} for the proof. 
% referenced in \cite{Trotter96,FK86}, 
\citet{FK86} improved this result to 
\begin{equation}
\label{DimLowerBound}
\dim(\Delta)\leq O(\Delta\log^2\Delta).
\end{equation}
On the other hand,  \citet*{EKT91} proved the lower bound, 
\begin{equation}
\label{DimUpperBound}
\dim(\Delta)\geq \Omega(\Delta\log \Delta).
\end{equation}
Both these proofs use probabilistic methods. The problem of narrowing the gap between \eqref{DimLowerBound} and \eqref{DimUpperBound} was described as ``an important topic for further research'' by \citet{EKT91}; \citet{Trotter96} speculated that the lower bound could be improved and wrote ``a really new idea will be necessary to improve the upper bound---if this is at all possible''; and \citet[page 52]{Wang15} described the problem as ``one of the most challenging (and probably quite difficult) problems in dimension theory.'' 

Our first contribution is the following result, which is the first improvement to the F{\"u}redi--Kahn upper bound in 30 years, and shows that  \eqref{DimUpperBound} is sharp to within a $(\log\Delta)^{o(1)}$ factor: 
\begin{equation}
\label{dimeqn}
\dim(\Delta) = \Delta\log^{1+o(1)}\Delta.
\end{equation}
 A more precise result is given below (see \cref{dim}).

\subsection{Boxicity and Degree}

We prove \eqref{dimeqn} via the notion of boxicity. The \emph{boxicity} of a (finite undirected) graph $G$, denoted by $\bx(G)$, is the minimum integer $d$, such that $G$ is the intersection graph of boxes in $\mathbb{R}^d$. Here a \emph{box} is a Cartesian product $I_1\times I_2\times\dots\times I_d$, where $I_i\subseteq \mathbb{R}$ is an interval for each $i\in[d]$. So a graph $G$ has boxicity at most $d$ if and only if there is a set of $d$-dimensional boxes $\{B_v:v\in V(G)\}$ such that $B_v\cap B_w\neq\emptyset $ if and only if $vw\in E(G)$. Note that a graph has boxicity 1 if and only if it is an interval graph. It is easily seen that every  graph has finite boxicity. 

Let $\bx(\Delta)$ be the maximum boxicity of a graph with maximum degree $\Delta$. It is easily seen that $\bx(2)=2$, and \citet{AC14} proved that $\bx(3)= 3$. \citet*{CFS08} proved the first  general upper bound of $\bx(\Delta)\leq 2\Delta^2+2$, which was improved to $\Delta^2+2$ by \citet{Esperet09}. A breakthrough was made by  \citet*{ABC11} who noted the following connection to poset dimension:
\begin{equation}
\label{BoxDim}
\tfrac12 \bx(\Delta-1) \leq \dim(\Delta) \leq 2 \bx(\Delta).
\end{equation}
Thus $\dim(\Delta)=\Theta(\bx(\Delta))$. For the sake of completeness, we prove \eqref{BoxDim} in \cref{bxdim,dimbx}. 
%Given a graph $G$, let $P$ be the poset on $V(G)\times\{0,1\}$ where $(u,i) \preceq_P (v,j)$ if and only if $i=0$ and $j=1$, and $u=v$ or $uv\in E(G)$.  \citet{ABC11} proved that $\frac12 \dim(P) -2 \leq \bx(G) \leq 2 \dim(P)$.  The comparability graph of $P$ has maximum degree $\Delta(G)+1$.  Thus $\bx(\Delta) \leq 2\dim(\Delta+1)$.  Conversely, \citet{ABC11} proved that if $G$ is the comparability graph of a poset $P$, then $\dim(P ) \leq 2 \bx(G )$, implying  $\dim(\Delta) \leq 2 \bx(\Delta)$. 

\citet{ABC11} concluded from \eqref{DimLowerBound}, \eqref{DimUpperBound} and \eqref{BoxDim}  that 
\begin{equation*}
\Omega(\Delta\log\Delta) \leq \bx(\Delta) \leq O(\Delta\log^2\Delta).
\end{equation*}
%Note that this lower bound disproved the earlier conjecture of \citet{CFS08} that $\bx(\Delta)\leq O(\Delta)$.  
We improve the upper bound, giving the following result, which is equivalent to \eqref{dimeqn}:
\begin{equation}
\label{degreeeqn}
\bx(\Delta) = \Delta\log^{1+o(1)}\Delta.
\end{equation}
 Again, a more precise result is given below (see \cref{degree}). 

%This paper studies the following natural questions: 
%\begin{itemize}
%\item What is the maximum boxicity of a graph with maximum degree $\Delta$, denoted by $\bx(\Delta)$? 
%\item What is the maximum boxicity of a graph with Euler genus $g$?
%%, denoted by $\bx(g)$? 
%\end{itemize}
%We answer the first question up to a $\log^{o(1)}\Delta$ factor. We answer the second question up to a $O(1)$ factor. 

%\subsection{Degree}

%\begin{thm}
%\label{degree}
%%For every graph $G$ with maximum degree $\Delta$,
%%$$\bx(G) \leq (322 + o(1) )\, \Delta\log(\Delta)\, (2e)^{\sqrt{\log \log\Delta}} \,\log\log \Delta.$$
%$$\bx(\Delta) \leq (180 + o(1) )\, \Delta\log(\Delta)\, (2e)^{\sqrt{\log \log\Delta}} \,\log\log \Delta.$$
%\end{thm}
%
%Note that $(2e)^{\sqrt{\log\log\Delta}}\log\log\Delta \leq \log^{o(1)}(\Delta)$. Thus $\bx(\Delta) \leq \Delta\log^{1+o(1)}\Delta$.

%Given a poset $P$, \citet{Kimble} proved there is a poset $Q$ of height 2 such that $\dim(P)\leq \dim(Q) \leq \dim(P)+1$ and both $P$ and $Q$ have the same maximum degree (see the final remark in \citep{FK86})

%\begin{thm}
%\label{dim}
%%For every poset $P$ whose comparability graph has maximum degree $\Delta$,
%%$$\dim(P) \leq (644 + o(1) )\, \Delta\log(\Delta)\, (2e)^{\sqrt{\log \log\Delta}} \log\log\Delta 
%%\leq \Delta\log^{1+o(1)}(\Delta).$$
%$$\dim(\Delta) \leq (360 + o(1) )\, \Delta\log(\Delta)\, (2e)^{\sqrt{\log \log\Delta}} \log\log\Delta 
%\leq \Delta\log^{1+o(1)}(\Delta).$$
%\end{thm}
%

\subsection{Boxicity and Genus}

Next consider the boxicity of graphs embeddable in a given surface. 
\citet{Scheinerman} proved that every outerplanar graph has boxicity at most 2. 
\citet{Thomassen86} proved that every planar graph has boxicity at most 3 (generalised to `cubicity' by \citet{FF11}). 
\citet{EJ13} proved that every toroidal graph has boxicity at most 7, improved to 6 by \citet{Esp17}.  
The \emph{Euler genus} of an orientable surface with $h$ handles is $2h$. 
The \emph{Euler genus} of  a non-orientable surface with $c$ cross-caps is $c$. 
The \emph{Euler genus} of a graph $G$ is the minimum Euler genus of a surface in which $G$ embeds (with no crossings). \citet{EJ13} proved that every graph with Euler genus $g$ has boxicity at most $5g+3$. 
\citet{Esperet16} improved this upper bound to $O(\sqrt{g}\log g)$ and also noted that there are graphs of
Euler genus $g$ with boxicity $\Omega(\sqrt{g\log g})$, which follows from the result of \citet{EKT91} mentioned above. See \citep{Esp17} for more on the boxicity of graphs embedded in surfaces.

The second contribution of this paper is to improve the upper bound to match the lower bound up to a constant factor (see \cref{genus}). We conclude that the maximum boxicity of a graph with Euler genus $g$ is 
\begin{equation}\label{genusupper}
\Theta(\sqrt{g \log g}).
\end{equation}
%
%\begin{thm}
%\label{genus}
%For every graph $G$ with Euler genus $g$, as $g\to\infty$, 
%$$\bx(G) \leq (12+o(1)) \sqrt{g \log g}.$$
%\end{thm}
Furthermore, the implicit constant in \eqref{genusupper} is not large: the upper bound in  \cref{genus} is $(12+o(1))\sqrt{g\log g}$.

\subsection{Boxicity and Layered Treewidth}

The third contribution of the paper
is to prove a new upper bound on boxicity in terms of layered
treewidth, which is a graph parameter recently introduced by
\citet{DMW17}
(see \cref{layeredtreewidth}). This generalises the known bound in terms of treewidth, and leads to generalisations of known results for graphs embedded in surfaces where each edge is in a bounded number of crossings. 

\subsection{Related Work} 

%Finally we mention several related works. 
The present paper can be considered to be part of a body of research connecting poset dimension and graph structure theory. Several recent papers \citep{Walczak17,JMW17,JMOW17,JMMTWW16,JMTWW17,JMW17a,MW17,ST14} show that structural properties of the cover graph of a poset lead to bounds on its dimension. Finally, we mention the following relationships between boxicity and chromatic number. Graphs with boxicity 1 (interval graphs) are perfect.  \citet{AG60} proved that graphs with boxicity 2 are $\chi$-bounded. But  \citet{Burling65} constructed triangle-free graphs with boxicity 3 and unbounded chromatic number. 

%Boxicity has also been studied for the following special classes of graphs: circular arc graphs \citep{ABC14,BC11a}, line graphs \citep{CMS11}, chordal bipartite graphs \citep{CFM11}, leaf powers \citep{CFM11a}, asteroidal triple free graphs \citep{BC10}, graphs of given treewidth \citep{CS07},	series-parallel graphs \citep{BCR06}, grid intersection graphs \citep{BHBPW93}, Halin graphs \citep{CFS09}, Mycielski graphs \citep{Kamibeppu13}, product graphs	\citep{CIMR15}, degenerate graphs \citep{ACM14}, graphs with given crossing number \citep{ACM14}, graphs with given chromatic number \citep{Kamibeppu}. See \citep{CFS10,BCJS16,ABC12,ACS10,ABC10,CR83,Krat94,FMP17} for algorithmic aspects of boxicity. \citet{ACS14} study lower bounds on boxicity. 

%Separation dimension = boxicity of line graph \citep{BCGMR14,BCGMR16,ABCMR15}
%		
%cubicity \citep{CMR16,ABC10,SA09,SDS09}

%fractional boxicity \citep{Kamibeppu15}
		
%Ferrers dimension and boxicity \citep{CG10}
			
%graphs with boxicity {$\leq 2$} \citep{QW90}

%poset boxicity \citep{TW87}
%		
%{$k$}-suitable sets of arrangements and boxicity \citep{CR84}

%%%%%%%%%%%%%%%%%%%%%%%%%		
\section{Tools}
\label{tools}

\citet{Roberts69} introduced boxicity and proved the following two fundamental results. 

\begin{lem}[\citep{Roberts69}]
\label{sumboxicity}
For all graphs $G,G_1,\dots,G_r$ such that $G=G_1\cap \dots\cap G_r$, 
\begin{equation*}
\bx(G)\leq\sum_{i=1}^r\bx(G_i).
\end{equation*}
\end{lem}

Note that \cref{sumboxicity} is proved trivially by taking Cartesian products.

\begin{lem}[\citep{Roberts69}] 
\label{numvertices}
Every $n$-vertex graph has boxicity at most $\floor{\frac{n}{2}}$. 
\end{lem}

Note that \citet{Trotter79} characterised those graphs for which equality holds in \cref{numvertices}. 

%A vertex-colouring of a graph is \emph{acyclic} if adjacent vertices receive distinct colours and every pair of colour classes induces a forest. 
%\citet{EJ13} made the following connection to boxicity. 
%
%\begin{lem}[\citep{EJ13}] 
%\label{acyclic}
%For $k\geq 2$, every acyclically $k$-colourable graph has boxicity at most $k(k-1)$. 
%\end{lem}

A graph is \emph{$k$-degenerate} if every subgraph has a vertex of degree at most $k$. Note that $1$-degenerate graphs (that is, forests) have boxicity at most 2, but $2$-degenerate graphs have unbounded boxicity, since the 1-subdivision of $K_n$ is $2$-degenerate and has boxicity $\Theta(\log\log n)$ \citep{EJ13}. \citet{ACM14} proved the following bound. Throughout this paper, all logarithms are natural unless otherwise indicated. 

%\footnote{Throughout this paper, all logarithms are natural. All leading term constants are made explicit, but little effort is made to optimise them.}

\begin{lem}[\citep{ACM14}] 
\label{degen}
Every $k$-degenerate graph on $n$ vertices 
has boxicity at most $(k + 2)  \ceil{2e \log n}$.
%$\bx(G) \leq \cub(G) \leq (k + 2)  \ceil{2e \log n}$.
\end{lem}

%The following lemma, due to  \citet{KT12}, is the starting point for our work on embedded graphs. 
The following lemma, due to  \citet{Esp17}, is the starting point for our work on embedded graphs. 

%\begin{lem}[\citep{KT12}] 
%\label{cutting}
%Every graph $G$ with Euler genus $g$ has a set $X$ of at most $1000g$ vertices such that $G -X$ is acyclically 7-colorable.
%\end{lem}

\begin{lem}[\citep{Esp17}] 
\label{newcutting}
Every graph $G$ with Euler genus $g$ has a set $X$ of at most $60g$ vertices such that $G -X$ has boxicity at most 5. 
\end{lem}

Let $[n]:=\{1,2,\dots,n\}$. For our purposes, a \emph{permutation} of a set $X$ is a bijection from $X$ to $[|X|]$. A set $\{\pi_1,\dots,\pi_p\}$ of permutations of a set $X$ is \emph{$r$-suitable} if for every $r$-subset $S$ of $X$ and for every element $x\in S$, there is permutation $\pi_i$ such that $\pi_i(x)<\pi_i(y)$ for all $y\in S\setminus\{x\}$. This definition was introduced by \citet{Dushnik50}; see  \citep{Spencer71a,Colbourn15,CJ17} for further results on suitable sets. 
%scrambling Rad03,Tarui08,Furedi96,
\citet{Spencer71a} attributes the following result  to Hajnal. We include the proof for completeness, and so that the dependence on $k$ is absolutely clear (since Spencer assumed that $k$ is fixed). 

\begin{lem}[\citep{Spencer71a}]\label{scrambling}
For every $k\ge2$ and $n\ge 10^4$ there is a $k$-suitable collection of permutations of size at most $k2^k\log\log n$.
\end{lem}

%\comment{DW: I think Kierstead (?) proved this result is almost best possible. Worth mentioning. Is this \citep{Kierstead96} the right paper?}

\begin{proof}
A sequence $S_1,\dots,S_r$ of subsets of $[s]$ is {\em $t$-scrambling} if for every set $I\subseteq [r]$ with $|I|\le t$ and every $A\subseteq I$, we have
\begin{equation}\label{venn}
\bigcap_{i\in A} S_i \cap\bigcap_{j\in I\setminus A} ([s]\setminus S_j)\ne\emptyset.
\end{equation}
For $s\ge t\ge1$, let $M(s,t)$ be the maximum cardinality of a $t$-scrambling family of subsets of $[s]$.  Note that $M(s,t)$ is monotone increasing in $s$ (and trivially $M(s,t)\le 2^s$).

\begin{claim}
Let $s\ge t\ge 1$.  If $m$ is a positive integer such that 
\begin{equation}\label{wenn}
2^t\binom mt(1-2^{-t})^s<1,
\end{equation} 
then $M(s,t)\ge m$.
\end{claim}

\begin{proof}[Proof of Claim.]
Choose subsets $S_1,\dots,S_m$ of $[s]$ independently and uniformly at random.  For any $t$-set $I\subseteq[m]$ and any $A\subseteq I$, the probability that \eqref{venn} is not satisfied is $(1-2^{-t})^s$.  
There are $\binom mt$ choices for $I$ and then $2^t$ choices for $A$, so taking a union bound, the probability that there is some pair $(I,A)$ such that \eqref{venn} is not satisfed  is at most the left hand side of \eqref{wenn}  
(note that it is enough to consider sets $I$ of size exactly $t$).  Since this is smaller than 1, we are done.
\end{proof}

Given $n$ and $k$, choose $s$ minimal so that $2^{M(s,k-1)}\ge n$.  Let $M:=M(s,k-1)$, let $S_1,\dots,S_M$ be a $(k-1)$-scrambling set of subsets of $[s]$, and let $Q_1,\dots,Q_n$ be distinct subsets of $[M]$.  
We define orders $<_1,\dots,<_s$ on $[n]$ as follows:
for $a,b\in[n]$, let $j(a,b)=\min(Q_a\triangle Q_b)$; then $a<_i b$ if either
\begin{itemize}
\item $i\in S_{j(a,b)}$ and $j(a,b)\in Q_a$; or 
\item $i\not\in S_{j(a,b)}$ and $j(a,b)\in Q_b$. 
\end{itemize}
(Note that if $S_i=[M]$ then this gives the lex order on the $Q_s$, and if $S_i=\emptyset$ it is reverse lex.)

Now $<_1,\dots,<_s$ is a $k$-suitable collection of orders on $[n]$.  This is straightforward, but a little tricky: given a set $B$ of $k$ elements of $[n]$ and $b\in B$, 
let $I:=\{\min (Q_h\triangle Q_b):h\in B\setminus\{ b\}\}$, let $A=I\cap Q_b$, and choose an element $i$ of the intersection
on the left hand side of \eqref{venn}.
Consider the order $<_i$.  It is enough to show that $b<_i h$ for each $h\in B\setminus\{b\}$.  Given $h\in B\setminus\{b\}$, let 
$q=j(b,h)=\min Q_b \triangle Q_h$.  If $q\in Q_b$ then $q\in A$ and so $i\in S_q$, and therefore $b<_i h$;
 if $q\not\in Q_b$ then $q\not\in A$ and so $i\not\in S_q$, and again $b<_i h$.

How big is $s$?  By the choice of $s$ and monotonicity, $M(s-1,k-1)<\log_2n$. The left hand side of \eqref{wenn} is less than
\begin{equation}\label{xenn}
(2em/t)^t\exp(-s2^{-t})=\left(\frac{2em\exp(-s/t2^t)}{t}\right)^t,
\end{equation}
and so 
$$M(s,t)\ge \frac{t}{2e}e^{s/t2^t},$$
as setting $m$ equal to the right hand side of this expression leaves \eqref{xenn} less than 1.
Thus, bounding $M(s-1,k-1)$, we have 
$$\frac{k-1}{2e}\exp\left(\frac{s-1}{(k-1)2^{k-1}}\right)-1\le M(s-1,k-1)\le \log_2n$$
and so
$$s\le 1+(k-1)2^{k-1}\log\left(\frac{2e}{k-1}\log_2(2n)\right),$$
which is at most $k2^k\log\log n$ for $k\ge2$ and $n\ge 10^4$.
\end{proof}

We will also use the Lov\'asz Local Lemma:

\begin{lem}[\citep{EL75}] 
\label{LLL}
Let $E_1,\dots,E_n$ be events in a probability space, each with probability at most $p$ and mutually independent of all but at most $D$ other events. If $4pD\leq 1$ then with positive probability, none of  $E_1,\dots,E_n$ occur. 
\end{lem}

For a graph $G$ and set $X\subseteq V(G)$, the graph $G[X]$ with vertex set $X$ and edge set $\{vw\in E(G):v,w\in X\}$ is called the subgraph of $G$ \emph{induced} by $X$. Let $G\langle X\rangle$ be the graph obtained from $G$ by adding an edge between every pair of non-adjacent vertices at least one of which is not in $X$. %Note that $G\langle X\rangle[X]\cong G[X]$

\begin{lem}
\label{completify}
$\bx(G\langle X\rangle) = \bx(G[X])$. 
\end{lem}

\begin{proof}
Given a $d$-dimensional box-representation of $G\langle X\rangle$, delete the boxes representing the vertices in $V(G)\setminus X$ to obtain a  $d$-dimensional box-representation of $G[X]$. Thus $\bx(G[X]) \leq \bx(G\langle X\rangle)$. Given a $d$-dimensional box-representation of $G[X]$, 
for every vertex $x$ in $V(G)\setminus X$, add a box with interval $\mathbb R$ in every dimension (so that it meets all other boxes). We obtain a  $d$-dimensional box-representation of $G\langle X\rangle$. Thus $ \bx(G\langle X\rangle) \leq \bx(G[X]) $.
 \end{proof}

For a graph $G$ and disjoint sets $X,Y\subseteq V(G)$, the graph $G[X,Y]$ with vertex set $X\cup Y$ and edge set $\{vw\in E(G):v\in X,w\in Y\}$ is called the bipartite subgraph of $G$ \emph{induced} by $X,Y$. For non-adjacent vertices $v\in X$ and $w\in Y$, we say $vw$ is a \emph{non-edge} of $G[X,Y]$. Let $G\langle X,Y\rangle$ be the graph obtained from $G$ by adding an edge between distinct vertices $v$ and $w$ whenever $v,w\in V(G)\setminus X$ or $v,w\in V(G)\setminus Y$.

%\begin{lem}
%\label{bipartitecompletify}
%
%\end{lem}

%\begin{lem}
%Let $G$ be a bipartite graph with bipartition $A,B$, such that every vertex in $A$ has degree at most $r$. 
%Let $G'$ be the graph obtained from $G$ by adding a clique on each of $A$ and $B$. 
%Then $\bx(G') \leq r 2^r \log\log |B|$. 
%\end{lem}
%
%\begin{proof}
%By \cref{scrambling}, there is a set of $r$-scrambling permutations $\pi_1,\dots,\pi_t$ of $B$.
%
%For each linear ordering $\sigma_i$ we introduce two dimensions, as illustrated in \cref{...}.  
%Represent each vertex $w\in B$ by the box with corners $(-\infty,+\infty)$ and $(2\sigma_i(w),2\sigma_i(w))$. 
%For each vertex $v\in A$, if $x$ and $y$ are respectively the leftmost and rightmost neighbours of $v$ in $\sigma_i$, 
%then represent $v$ by the box with corners $(2\sigma_i(x)-1,2\sigma_i(y)+1)$ and $(+\infty,-\infty)$. 
%Observe that $A$ and $B$ are both cliques in this representation. 
%If $vw\in E(G)$ and $v\in A$ and $w\in B$, then the box representing $v$ intersects the box representing $w$. 
%For a non-edge $vw$ with $v\in A$ and $w\in B$, the box representing $v$ intersects the box representing $w$ if and only if $\sigma_i$ catches $vw$. 
%Since for every non-edge $vw$, some $\sigma_i$ does not catch $vw$, 
%the boxes representing $v$ and $w$ do not intersect. Thus
%$\bx(\widehat{G} [A,B]) \leq 2t \leq 2+ 3 (k+1) \log n$. 
%\end{proof}

For the sake of completeness we now establish the relationship between $\bx(\Delta)$ and $\dim(\Delta)$ from \eqref{BoxDim}. In fact, \cref{bxdim} is a slight strengthening. 

\begin{prop}
\label{bxdim}
$\bx(\Delta) \leq \dim(\Delta+1)$. 
\end{prop}

\begin{proof}
Let $G$ be a graph with maximum degree $\Delta$. Let $P$ be the poset on $V(G)\times\{0,1\}$ where $(u,i) \preceq_P (v,j)$ if and only if $i=0$ and $j=1$, and $u=v$ or $uv\in E(G)$. Let $d:=\dim(P)$. Let $\preceq_1,\dots,\preceq_d$ be total orders on $V(G)\times\{0,1\}$ whose intersection is $P$. For each vertex $v$ of $G$, let $B_v$ be the box $[x_1(v,0),x_1(v,1)]\times\dots\times [x_d(v,0),x_d(v,1)]$, where  $x_i(v,j)$ is the position of $(v,j)$ in $\preceq_i$. Let $v,w$ be distinct vertices in $G$. If $vw\in E(G)$ then $(v,0),(w,0) \prec_i (v,1),(w,1)$ for each $i\in[d]$, implying that $B_v\cap B_w\neq\emptyset$. If $vw\not\in E(G)$, then $(v,1)\prec_i (w,0)$ for some $i\in[d]$, implying that $(v,0) \prec_i (v,1)\prec_i (w,0) \prec_i (w,1)$ and $B_v\cap B_w=\emptyset$. Thus $\{B_v:v\in V(G)\}$ is a box representation for $G$, and $\bx(G)\leq \dim(P)$. Since the comparability graph of $P$ has maximum degree $\Delta+1$, we have $\bx(\Delta) \leq \dim(\Delta+1)$. 
\end{proof}

The next lemma uses the following characterisation of poset dimension due to \citet{FK86}: Let $P$ be a poset with ground set $X$. Then $\dim(P)$ equals the minimum integer $k$ for which there are total orders $\preceq_1,\dots,\preceq_k$ of $X$, with the property that for all incomparable pairs $(x,y)$ of $P$ there is an order $\preceq_i$ such that $x \prec_i z$ for all $z\in X$ with $y \preceq_P z$. 

\begin{prop}
\label{dimbx}
$\dim(\Delta) \leq 2\bx(\Delta)$.
\end{prop}

\begin{proof}
Let $P$ be a poset, such that comparability graph $G$ of $P$ has maximum degree $\Delta$. Let $\{ [\ell_1(x),r_1(x)] \times\dots\times [\ell_d(x),r_d(x)] :x\in V(G)\}$ be a $d$-dimensional box representation of $G$, where $d:=\bx(G)\leq\bx(\Delta)$. For $i\in[d]$, let $\preceq'_i$ be the total order on $X$ determined by $(\ell_i(x):x\in X)$ ordered right-to-left, and let $\preceq''_i$ be the total order on $X$ determined by $(r_i(x):x\in X)$ ordered left-to-right. Let $(x,y)$ be an incomparable pair of $P$. Thus $xy\not\in E(G)$, and $r_i(x) < \ell_i(y)$ or $r_i(y) < \ell_i(x)$ for some $i\in[d]$. 
First suppose that $r_i(x) < \ell_i(y)$. Consider $z\in X$ such that $y\preceq_P z$. Thus $y=z$ or $yz\in E(G)$, 
implying  $r_i(z) \geq \ell_i(y)$. Hence $r_i(x) < \ell_i(y) \leq r_i(z)$. By construction, $x \prec''_i z$ as desired. 
Now assume that $r_i(y) < \ell_i(x)$. Consider $z\in X$ such that $y\preceq_P z$. Thus $y=z$ or $yz\in E(G)$, implying  $r_i(y) \geq \ell_i(z)$. Hence $\ell_i(z) \leq r_i(y) < \ell_i(x)$. By construction, $x \prec'_i z$, as desired.  By the above characterisation, $\dim(P) \leq 2d = 2\bx(G)$, and $\dim(\Delta) \leq 2\bx(\Delta)$.
\end{proof}

%%%%%%%%%%%%%%%%%%%%%%%%%		
\section{Bounded Degree}
\label{boundeddegree}

The first ingredient in our proof is the following colouring result that bounds the number of monochromatic neighbours of each vertex. A very similar result was proved by \citet{HMR97}; they required the additional property that the 
colouring is proper, but had $k=\max\{(d+1)\Delta,e^3 \Delta^{1+1/d}/d\}$, which is too much for our purposes.

\begin{lem}
\label{StepZero}
For every graph $G$ with maximum degree $\Delta>0$ and for all integers $d\geq 1$ and $k\geq  \frac{(4d+4)^{1/d}e}{d} \Delta^{1+1/d}$,  there is a partition $V_1,\dots,V_k$ of $V(G)$, such that $|N_G(v)\cap V_i| \leq d$ for each $v\in V(G)$ and $i\in[k]$. 
\end{lem}

\begin{proof} 
Colour each vertex of $G$ independently and randomly with one of $k$ colours. 
Let $V_1,\dots,V_k$ be the corresponding colour classes. 
For each set $S$ of exactly $d+1$ vertices in $G$, such that $S\subseteq N_G(v)$ for some vertex $v\in V(G)$, introduce an event which holds if only if $S\subseteq V_i$ for some $i\in[k]$. Each such event has probability $p:=k^{-d}$. The colour on one vertex affects  at most $\Delta\binom{\Delta-1}{d}$  events. Thus each event is mutually independent of all but at most $D$ other events, where 
\begin{equation*}
D:=(d+1) \Delta\binom{\Delta-1}{d}\leq (d+1) \Delta\left( \frac{e \Delta}{d} \right)^d 
= (d+1) \left( \frac{e}{d} \right)^d \Delta^{d+1}.
\end{equation*}
It follows that $4pD\leq 1$. By \cref{LLL}, with positive probability, no event occurs. Thus the desired partition exists.  
\end{proof}

Note that an example in \citep{HMR97}, due to Noga Alon, shows that the value of $k$ in \cref{StepZero} is within a constant factor of optimal. 
%As an aside, we now show that for fixed $d$, the number of colours in \cref{StepZero} is within a constant factor of optimal. Let $G$ be the graph with vertex set $[n]^{d+1}$, where $vw\in E(G)$ if and only if $v$ and $w$ have at least one coordinate in common. Thus $\Delta(G) \leq (d+1)n^d$. Consider a $k$-colouring of $G$ in which each colour is assigned to at most $d$ neighbours of each vertex. Suppose that distinct vertices $v_1,\dots,v_{d+1}$ are assigned the same colour. Let $w$ be the vertex whose $i$-th coordinate equals the $i$-th coordinate of $v_i$. Thus $v_1,\dots,v_{d+1}\in N_G(w)$, which is a contradiction. Thus at most $d$ vertices are assigned the same colour, and the number of colours is at least $n^{d+1}/d\geq (\Delta/(d+1))^{(d+1)/d}/d$. 
\cref{StepZero}  leads to our next lemma. A similar result was used by \citet{FK86} in their work on poset dimension. 
%ABCMR15

\begin{cor}
\label{StepOne}
For every graph $G$ with maximum degree $\Delta\geq 2$ and for all integers $d\geq 100 \log\Delta$ and  $k\geq \frac{3\Delta}{d}$,  there is a partition $V_1,\dots,V_k$ of $V(G)$, such that $|N_G(v)\cap V_i| \leq d$ for each $v\in V(G)$ and $i\in[k]$. 
\end{cor}

\begin{proof}
Since $d\geq 100\log\Delta \geq 69$, we have $(4d+4)^{1/d}e \leq 2.95$ 
and $d\geq \log^{-1}(\frac{3}{2.95})\log\Delta$ 
and $\Delta^{1/d}\leq\frac{3}{2.95}$. Thus
$\frac{(4d+4)^{1/d}e}{d} \Delta^{1+1/d} \leq \frac{2.95}{d} \Delta^{1+1/d} \leq \frac{3\Delta}{d} \leq k$. 
The result follows from \cref{StepZero}.
\end{proof}

%\begin{proof}
%Colour each vertex uniformly at random with one of $k$ colours. 
%Let $V_1,\dots,V_k$ be the corresponding colour classes. For $v\in V(G)$, let $E_{v}$ be the event that 
%$|N_G(v)\cap V_i| \geq d$ for some $i\in[k]$. 
%Then $\mathbb{P}(E_v) = \binom{\deg(v)}{d} k^{1-d} \leq \binom{\Delta}{d} k^{1-d}$.  
%Each event is mutually independent of all but at most $\Delta^2$ other events. 
%Since $d\geq 4 + 3\log \Delta$, we have $32  \Delta^3 < e^d < (\frac{8}{e})^{d} $. 
%Thus $4 e ^d  \Delta^3 < 8^{d-1} $ and $4 e ^d  \Delta^{d+2} < (8 \frac{\Delta}{d}) ^{d-1}  d^d \leq  k^{d-1}  d^d$. 
%Hence  $4\binom{\Delta}{d} k^{1-d} \Delta^2 < 4 (\frac{e \Delta}{d})^d  k^{1-d} \Delta^2 \leq 1$.
%By \cref{LLL}, with positive probability no event $E_v$ occurs. Therefore, the desired colouring exists. 
%\end{proof}

The next lemma is a key new idea in our proof. Its proof is a straightforward application of the Lov\'asz Local Lemma.

\begin{lem}
\label{StepTwo}
Let $G$ be a bipartite graph with bipartition $\{A,B\}$, where vertices in $A$ have degree at most $d$ and vertices in $B$ have degree at most $\Delta$. 
Let $r,t,\ell$ be positive integers such that 
\begin{equation*}
\ell \geq e \left(\frac{ed}{r+1}\right)^{1+1/r} \text{ and } t\geq\log( 4 d \Delta). 
\end{equation*}
Then there exist $t$ colourings $c_1,\dots,c_t$ of $B$, each with $\ell$ colours, such that for each vertex $v\in A$, for some colouring $c_i$, 
each colour is assigned to at most  $r$ neighbours of $v$ under $c_i$. 
\end{lem}

\begin{proof}
For $i\in[t]$ and for each vertex $w\in B$, let $c_i(w)$ be a random colour in $[\ell]$. 
Let $X_v$ be the event that for each $i\in[t]$, some set of $r+1$ neighbours of $v$ are monochromatic under $c_i$. 
The probability that there is a monochromatic set of at least $r+1$ neighbours of $v$ under $c_i$ is 
\begin{equation*}
\binom{\deg(v)}{r+1} \ell^{-r} \leq 
\binom{d}{r+1}\ell^{-r} \leq 
\left( \frac{ed}{r+1} \right)^{r+1}\ell^{-r} \leq e^{-1}.
\end{equation*}
%$$\left( \frac{ed}{r+1} \right)^{r+1}\ell^{-r} \leq e^{-1}$$
%$$e \left( \frac{ed}{r+1} \right)^{r+1} \leq \ell^r$$
%$$e^{1/r} \left( \frac{ed}{r+1} \right)^{1+1/r} \leq \ell$$
Thus $\mathbb{P}(X_v) \leq e^{-t}$. Observe that $X_v$ is mutually independent of all but at most $d\Delta$ other events. 
By assumption, $ 4 e^{-t} d \Delta \leq 1$. 
By \cref{LLL}, with positive probability no event $X_v$ occurs. Therefore, the desired colourings exists. 
\end{proof}

\begin{lem}
\label{unbalancedbipartite}
Let $G$ be a bipartite graph with bipartition $\{A,B\}$, where vertices in $A$ have degree at most $d$ and vertices in $B$ have degree at most $\Delta$, for some $\Delta\geq d\geq 2$. 
Let $G'=G\langle A,B\rangle$ be the graph obtained from $G$ by adding a clique on $A$ and a clique on $B$. 
Then, as $d\to\infty$, 
\begin{equation*}
\bx(G') \leq 
(60 +o(1) ) \,d \,\log( d \Delta)\, \log\log( \Delta ) \,  (2e)^{\sqrt{\log d}} .
\end{equation*}
\end{lem}

\begin{proof}
Let $r:= \ceil{ \sqrt{\log d}}$ and $\ell:=\ceil{ e \left(\frac{ed}{r+1}\right)^{1+1/r}}$  and $t:= \ceil{ \log( 4 d \Delta)}$.  As $d\to\infty$, we may assume that $d$ is large.

% V(G) \ V_i = A
% V_i = B

By \cref{StepTwo}, there exist $t$ colourings $c_1,\dots,c_t$ of $B$, each with $\ell$ colours, such that for each vertex $v\in A$, 
for some colouring $c_j$, each colour is assigned to at most $r$ neighbours of $v$. Let $\{A_j : j \in[t] \}$ be a partition of $A$ such that for each $v\in A_j$, at most $r$ neighbours of $v$ are assigned the same colour under $c_j$.  Assume that $[\ell]$ is the set of colours used by each $c_j$. 

Our aim is to construct a box representation for each of the graphs $G\langle A_i,B\rangle$, and then take their intersection using \cref{sumboxicity}. In light of  \cref{completify}, it is enough to concentrate on the subgraphs $G[A_i\cup B]$.   To handle $G[A_i\cup B]$, we further decompose $B$  according to $c_1,\dots,c_t$ as follows. For each $j\in[t]$ and each colour $\alpha\in[\ell]$, let $B_{j,\alpha}:= \{ w \in B : c_j(w) = \alpha \}$. 
Let $G_{j,\alpha} := G\langle A_j,B_{j,\alpha} \rangle$.
Note that $G'= \bigcap_{j,\alpha} G_{j,\alpha}$. 

We now bound the boxicity of $G_{j,\alpha}$. 
Let $H$ be the graph with vertex set $B_{j,\alpha}$, where distinct vertices $x,y\in B_{j,\alpha}$ are adjacent in $H$ whenever $x$ and $y$ have a common neighbour in $A_j$. 
Since each vertex in $A_j$ has at most $r$ neighbours in $B_{j,\alpha}$, the graph $H$ has maximum degree at most $r\Delta$. 
Thus $\chi(H) \leq h:= r\Delta+1$. 
Let $X_{1},\dots,X_{h}$ be the colour classes in a proper colouring of $H$. 
For $q\in[h]$, let $\overrightarrow{X_q}$ denote an arbitrary linear ordering of $X_q$. 
Let $\overleftarrow{X_q}$ be the reverse of $\overrightarrow{X_q}$.
Since we may assume that $d$ is large, 
\cref{scrambling} shows that there exists a set of $(r+1)$-suitable permutations $\pi_1,\dots,\pi_p$ of $[h]$ for some $p\leq (r+1)2^{r+1}\log\log(h)$. 

For each $a\in[p]$, we introduce two 2-dimensional representations of $G_{j,\alpha}$. 
Let $\sigma_a$ be the ordering
$\overrightarrow{X_{\pi_a(1)}} , \overrightarrow{X_{\pi_a(2)}} , \dots,\overrightarrow{X_{\pi_a(h)}}$ of $B_{j,\alpha}$. 
Similarly, let $\sigma'_a$  be the ordering
$\overleftarrow{X_{\pi_a(1)}} , \overleftarrow{X_{\pi_a(2)}} , \dots,\overleftarrow{X_{\pi_a(h)}}$ of $B_{j,\alpha}$. 
For each vertex $x$ in $B$, say $x$ is the $b_x$-th vertex in  $\sigma_a$ and $x$ is the $b'_x$-th vertex in  $\sigma'_a$. 
Then represent $x$ by the box with corners $(-\infty,+\infty)$ and $(2b_x,2b_x)$. 
For each vertex $v\in A_j$, 
if $v$ has no neighbours in $B$, then represent $v$ by the  point $(2|B|,-2|B|)$; 
otherwise, if $x$ is the leftmost neighbour of $v$ in $\sigma_a$  and 
$y$ is the rightmost neighbour of $v$ in $\sigma_a$, then represent $v$ by the  box with corners $(\infty,-\infty)$ and $(2 b_x-1,2b_y+1)$, as illustrated in \cref{Gja1}. 
Now, in two new dimensions introduce the following representation. 
Represent each $x$ in $B_{j,\alpha}$ by the box with corners $(-\infty,+\infty)$ and $(2b'_x,2b'_x)$. 
For each vertex $v\in A_j$, 
if $v$ has no neighbours in $B$, then represent $v$ by the  point $(2|B|,-2|B|)$; 
otherwise,  
if $x$ is the leftmost neighbour of $v$ in $\sigma'_a$  and 
$y$ is the rightmost neighbour of $v$ in $\sigma'_a$, 
then represent $v$ by the  box with corners $(\infty,-\infty)$ and $(2 b'_x-1,2b'_y+1)$. 
In each of these four dimensions, add every vertex in 
$V(G)\setminus(B_{j,\alpha}\cup A_j)$ with interval $\mathbb{R}$. 
Observe that $A$ and $B$ are both cliques in this representation. 

%\begin{figure}[ht]
%\includegraphics{Gja}
%\caption{\label{Gja} Representation of $G_{j,\alpha}$ with respect to $\sigma_a$.}
%\end{figure}
%
\begin{figure}[ht]
\includegraphics{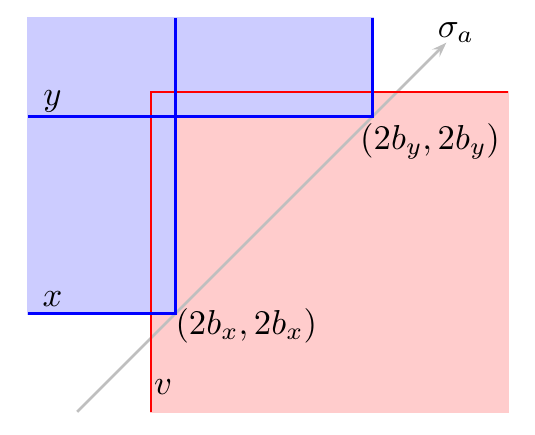}
\caption{\label{Gja1} Representation of $G_{j,\alpha}$ with respect to $\sigma_a$.}
\end{figure}

By construction, for every edge $vw$ of $G_{j,\alpha}$ the boxes of $v$ and $w$ intersect in every dimension. 
Now consider a non-edge $zv$ of $G_{j,\alpha}$ with $z\in B_{j,\alpha}$ and $v\in A_j$. 
Let $C$ be the set of integers $q\in [h]$ such that some neighbour of $v$ is in $X_q$. 
Thus $|C|\leq r$. Say $z$ is in $X_{q'}$. 
First suppose that $q'\not\in C$. 
Since $|C\cup\{q'\}| \leq r+1$, for some permutation $\pi_a$, we have $\pi_a(q') <  \pi_a(q)$ for each $q\in C$. 
Let $x$ be the leftmost neighbour of $v$ in $\sigma_a$. 
Thus $b_z< b_x$, and in the first 2-dimensional representation corresponding to $\pi_a$, 
the right-hand-side of the box representing $z$ is to the left of the left-hand-side of the box representing $v$, as illustrated in \cref{Gja2}(a). 
Thus the boxes representing $v$ and $z$ do not intersect. 
Now assume that $q' \in C$. 
By construction, there is exactly one neighbour $x$ of $v$ in $X_{q'}$. 
Since $|C| \leq r$, for some permutation $\pi_a$, we have $\pi_a(q') \leq \pi_a(q)$ for each $q\in C$. 
If $z<x$ in $\overrightarrow{X_q}$, then $b_z< b_x$, and as argued above and illustrated in \cref{Gja2}(b), 
the boxes representing $v$ and $z$ do not intersect. 
Otherwise,  $z<x$ in $\overleftarrow{X_q}$. Then $b'_z< b'_x$, and 
in the second 2-dimensional representation corresponding to $\pi_a$, 
the right-hand-side of the box representing $z$ is to the left of the left-hand-side of the box representing $v$. 
Hence the boxes representing $v$ and $z$ do not intersect. 
Therefore $\bx(G_{j,\alpha}) \leq 4p$. 

\begin{figure}[ht]
\includegraphics{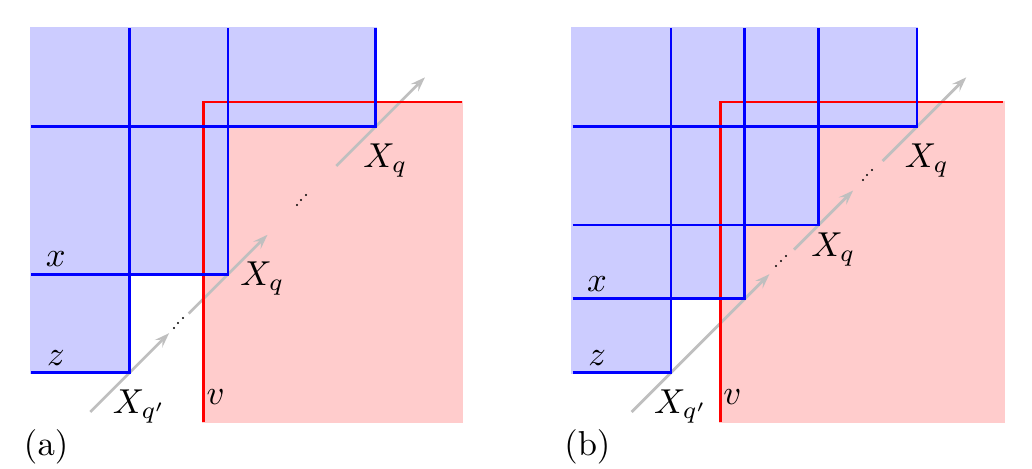}
\caption{\label{Gja2} Proof for the representation of $G_{j,\alpha}$.}
\end{figure}

By \cref{sumboxicity}, 
\begin{align*}
\bx(G')   & \leq 
4 t\ell p\\
&  \leq 
  4 \ceil{ \log( 4 d \Delta)} \, \ceil*{ e \left(\frac{ed}{r+1}\right)^{1+1/r}} \, (r+1)2^{r+1} \, \log\log( r\Delta+1 ).
  \end{align*}
  Since $\log( 4 d \Delta) \leq (1+o(1)) \log(d\Delta)$ and $e^{1+1/r}\leq (1+o(1)) e$ and $\log(r\Delta+1)\leq (1+o(1)) \log(\Delta)$, 
\begin{align*}
\bx(G')  
& \leq  
(8e^2 +o(1) ) \log(d \Delta) \, \left(\frac{d}{r+1}\right)^{1+1/r} \, (r+1)2^{r} \, \log\log( \Delta )\\
& \leq 
(60 +o(1) ) \log(d \Delta) \, \log\log( \Delta )\, d^{1+1/r} \,  \frac{2^{r}}{(r+1)^{1/r}} \\
& \leq 
(60 +o(1) ) \, d \,\log( d\Delta)\, \log\log( \Delta ) \, \left( d^{1/r} \, 2^{r} \right)\\
 & \leq 
(60 +o(1) )\, d \, \log(d \Delta)\, \log\log( \Delta ) \, (2e)^{\sqrt{\log d}} . \qedhere
\end{align*}
\end{proof}

We now prove our first main result.

\begin{thm}
\label{degree}
For every graph $G$ with maximum degree $\Delta$, as $\Delta\to\infty$, 
\begin{equation*}
\bx(G) \leq (180+ o(1) )\, \Delta\log(\Delta)\, (2e)^{\sqrt{\log \log\Delta}} \,\log\log \Delta.
\end{equation*}
\end{thm}

\begin{proof} 
Let $d:= \ceil{100\log\Delta}$ and $k := \ceil{ \frac{3\Delta}{d}}$. 
By \cref{StepOne},  there is a partition $V_1,\dots,V_k$ of $V(G)$, such that $|N_G(v)\cap V_i| \leq d$ for each $v\in V(G)$ and $i\in[k]$. 
Note that 
\begin{equation*}
\bx(G) = \bigcap_i G\langle{V_i}\rangle \cap G\langle{V_i,V(G)\setminus V_i}\rangle .
\end{equation*}
Since $G[V_i]$ has maximum degree at most $d$, by the result of \citet{Esperet09}, the graph $G[V_i]$ has boxicity at most $d^2+2$. 
By \cref{completify}, 
\begin{equation*}
\bx(G\langle V_i\rangle) \leq d^2+2.
\end{equation*}
Let $G_i := G[ V_i, V(G)\setminus V_i]$. 
Every vertex in $V(G)\setminus V_i$ has degree at most $d$ in $G_i$. 
Let $G'_i$ be obtained from $G_i$ by adding a clique on $V_i$ and a clique on $V(G)\setminus V_i$. 
By \cref{unbalancedbipartite,completify} and since $\log(d\Delta)\leq (1+o(1)) \log\Delta$, 
\begin{equation*}
\bx( G\langle V_i,V(G)\setminus V_i\rangle ) \leq 
\bx(G'_i)  \leq  (60 + o(1) )\, d\,  \log(  \Delta) \, \log\log( \Delta )\,  (2e)^{\sqrt{\log d}}  .
\end{equation*}
Applying \cref{sumboxicity} again, 
\begin{align*}
\bx(G)
 & \leq k (d^2+2) + (60 + o(1) )\, k \,d\, \log(  \Delta) \,\log\log( \Delta ) \, (2e)^{\sqrt{\log d}}  \\
 & \leq (9+o(1)) ( \Delta \log \Delta) + (180 + o(1) )\,\Delta \log(\Delta) \,\log\log( \Delta )\, (2e)^{\sqrt{\log \log\Delta}}  \\
 & \leq (180 + o(1) \, \Delta\log(\Delta)\, \log\log( \Delta )\, (2e)^{\sqrt{\log \log\Delta}} .
 \qedhere
\end{align*}
\end{proof}

Since $(2e)^{\sqrt{\log\log\Delta}}\log\log\Delta \leq \log^{o(1)}\Delta$, 
\cref{degree} implies \eqref{degreeeqn}.   More precisely, 
%$$\bx(\Delta)= \Delta (\log \Delta) e^{\Theta(\sqrt{\log\log\Delta})}.$$
$$\bx(\Delta)= \Delta (\log \Delta) e^{O(\sqrt{\log\log\Delta})}.$$

\cref{degree} and the result of \citet{ABC11} mentioned in \cref{Introduction} imply the following quantitative version of \eqref{dimeqn}.

\begin{thm}
\label{dim}
For every poset $P$ whose comparability graph has maximum degree $\Delta$, as $\Delta\to\infty$, 
\begin{equation*}
\dim(P) \leq (360 + o(1) )\, \Delta\log(\Delta)\, (2e)^{\sqrt{\log \log\Delta}} \log\log\Delta.
\end{equation*}
\end{thm}

Again, with \eqref{DimUpperBound},  this gives
%$$\dim(\Delta)= \Delta (\log \Delta) e^{\Theta(\sqrt{\log\log\Delta})}.$$
$$\dim(\Delta)= \Delta (\log \Delta) e^{O(\sqrt{\log\log\Delta})}.$$

%%%%%%%%%%%%%%%%%%%%%%
\section{Euler Genus}
\label{eulergenus}

We now prove our second main result.

\begin{thm}
\label{genus}
For every graph $G$ with Euler genus $g$, as $g\to\infty$.
  \begin{equation*}
\bx(G)\leq (12 + o(1) ) \sqrt{g\log g}.
\end{equation*}
 \end{thm}

\begin{proof}
%By \cref{cutting}, $G$ contains a set $X$ of at most $1000g$ vertices such that $G -X$ is acyclically 7-colorable. 
%By \cref{acyclic}, $\bx( G-X) \leq 42$. 
By \cref{newcutting}, $G$ contains a set $X$ of at most $60g$ vertices such that $\bx(G -X)\leq 5$.
First suppose that $|X|<10^4$. 
Deleting one vertex reduces boxicity by at most 1. Thus $\bx(G)\le\bx(G-X)+|X| \leq 10^4+5$, and we are done since $g\rightarrow\infty$. 
Now assume that $|X| \geq 10^4$. 
%If $X=\emptyset$ then we are done. Now assume that $X\neq\emptyset$. 

Let  $G_1:=G\langle V(G)\setminus X \rangle$. 
Let $Y$ be the set of vertices in $G-X$ with exactly one or exactly two neighbours in $X$. 
Let $G_2:= \langle X,Y \rangle$. 
%We may assume that $|X|$ is large enough for \cref{scrambling} to apply (if not then, as $g\to\infty$,  note that deleting a vertex reduces boxicity by at most 1, and so $\bx(G)\le\bx(G-X)+|X|$, and we are done). 
By \cref{scrambling}, there is a 3-suitable set of permutations $\pi_1,\dots,\pi_p$ of $X$ for some $p\leq 24 \log\log |X|$. 
For each $\pi_i$ we introduce two dimensions, as illustrated in \cref{G2}.  
Represent each vertex $w\in X$ by the box with corners $(-\infty,+\infty)$ and $(2\pi_i(w),2\pi_i(w))$. 
For each vertex $v\in Y$, 
if $x$ and $y$ are respectively the leftmost and rightmost neighbours of $v$ in $\pi_i$, 
then represent $v$ by the box with corners $(2\pi_i(x)-1,2\pi_i(y)+1)$ and $(+\infty,-\infty)$. 
Observe that $X$ and $Y$ are both cliques in this representation. 
If $vw\in E(G)$ and $v\in Y$ and $w\in X$, then the box representing $v$ intersects the box representing $w$. 
Consider  a non-edge $vz$ with $v\in Y$ and $z\in X$. 
Since $\pi_1,\dots,\pi_p$ is 3-suitable and $\deg_X(v)\leq 2$, for some $i$, we have $\pi_i(z) < \pi(x)$ for each $x\in N_G(v)$. 
Thus, for the 2-dimensional representation defined with respect to $\pi_i$, 
the boxes representing $v$ and $z$ do not intersect. 

\begin{figure}[ht]
\centering
\includegraphics{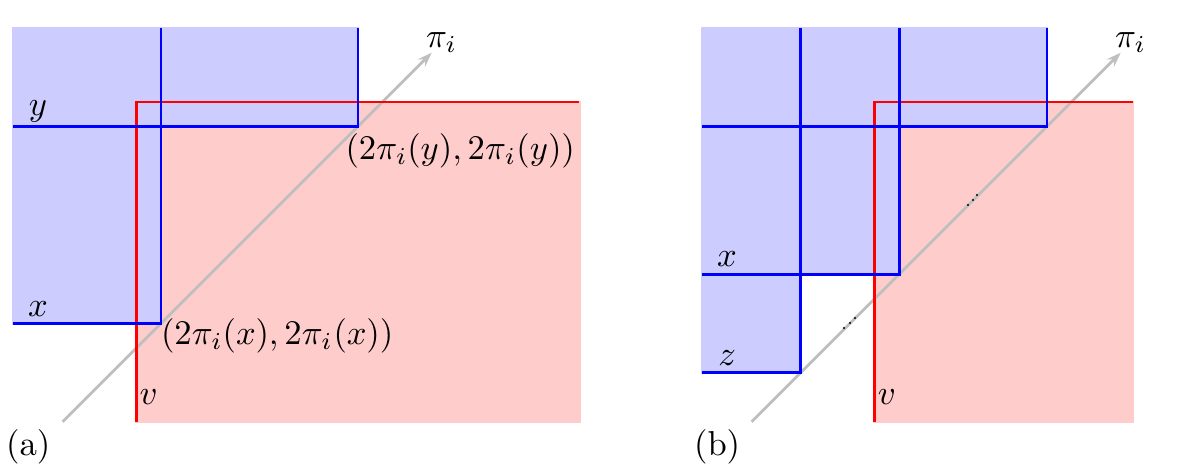}
\caption{\label{G2} Representation of $G_2$ with respect to $\pi_i$.}
\end{figure}

Add each vertex in $G-(X\cup Y)$ to every dimension with interval $\mathbb{R}$. 
We obtain a box representation of $G_2$. 
Thus $\bx(G_2) \leq 48 \log\log |X| \leq 48 \log\log (1000g)$. 

Let $Z$ be the set of vertices in $G-X$ with at least three neighbours in $X$. 
Let $G_3:= G \langle X\cup Z\rangle$.
% be the graph with vertex set $V(G)$, where $vw\in E(G_3)$ whenever $vw\in G[X\cup Z]$ or at least one of $v,w$ is in $V(G)-(X\cup Z)$. 
Observe that $G=G_1\cap G_2\cap G_3$. 

To bound $\bx(G_3)$, we first bound $\bx(H)$, where $H:=G[X\cup Z]$. 
The number of edges in $G[X,Z]$ is least $3|Z|$ and at most $2(|X| + |Z| + g-2)$ by Euler's formula. 
Thus $|Z| < 2(|X| + g)$, implying $|X\cup Z| < 3002g$. 
Let $n:= |V(H)| < 3002g$. 
Let $v_1,\dots,v_n$ be an ordering of $V(H)$, where $v_i$ has minimum degree in $H[\{v_i,\dots,v_n\}]$. 
Define $k:= 7 + \ceil{ \sqrt{g / \log g} }$. 
Let $i$ be minimum such that $v_i$ has degree at least $k$ in $H[\{v_i,\dots,v_n\}]$.
If $i$ is defined, then let $A:=\{v_1,\dots,v_{i-1}\}$ and $B:=\{v_i,\dots,v_n\}$, otherwise let $A:=V(H)$ and $B:=\emptyset$.

Observe that $H= H\langle A\rangle \cap H\langle B\rangle \cap H\langle A,B\rangle $. 

By construction, $H[A]$ is $k$-degenerate and has at most $n$ vertices. 
By \cref{degen}, $\bx(H[A]) \leq (k+2) \ceil{2e \log n}$. 
By \cref{completify}, $H\langle A\rangle \leq (k+2) \ceil{2e \log n}$. 

By construction, $H[B]$ has minimum degree at least $k$. 
The number of edges in $H[B]$ is at least $\frac{1}{2} k|B|$ and at most $3(|B|+g-2)$, implying
$ (\frac{k}{2} -3)  |B| < 3g$. 
By \cref{numvertices}, 
$\bx(H[B]) \leq \frac{|B|}{2} < \frac{3g}{k-6}$. 
By \cref{completify}, $H\langle B\rangle < \frac{3g}{k-6}$. 

Now consider $H[A,B]$. 
By construction, every vertex in $A$ has degree at most $k$ in $H[A,B]$. 
A permutation $\sigma$ of $B$ \emph{catches} a non-edge $vw$ of $H[A,B]$ with $v\in A$ and $w\in B$ if there are edges $vx,vy$ in $H[A,B]$, such that $w$ is between $x$ and $y$ in $\sigma$. 
Let $t:= \ceil{\frac32 (k+1) \log n}$. 
Let $\sigma_1,\dots,\sigma_t$ be random permutations of $B$. 
For each non-edge $vw$ of $H[A,B]$, 
the probability that $\sigma_i$ catches $vw$ equals $1-\frac{2}{\deg(v)+1} \leq e^{-2/(k+1)}$. 
Thus, the probability that every $\sigma_i$ catches $vw$ is at most $e^{-2t/(k+1)} \leq  n^{-3}$. 
Since the number of non-edges is at most $n^2$, by the union bound, 
 the probability that for some non-edge $vw$, every $\sigma_i$ catches $vw$ is at most $n^{-1}<1$. 
Hence, with positive probability, for every non-edge $vw$, some $\sigma_i$ does not catch $vw$.
Therefore, there exists permutations  $\sigma_1,\dots,\sigma_t$ of $B$, such that for every non-edge $vw$, some $\sigma_i$ does not catch $vw$. 

For each permutation $\sigma_i$ we introduce two dimensions. 
Represent each vertex $w\in B$ by the box with corners $(-\infty,+\infty)$ and $(2\sigma_i(w),2\sigma_i(w))$. 
For each vertex $v\in A$, 
if $v$ has no neighbours in $B$ then represent $v$ by the point $(2|B|,-2|B|)$; 
otherwise, if $x$ and $y$ are respectively the leftmost and rightmost neighbours of $v$ in $\sigma_i$, 
then represent $v$ by the box with corners $(2\sigma_i(x)-1,2\sigma_i(y)+1)$ and $(+\infty,-\infty)$. 
Observe that $A$ and $B$ are both cliques in this representation. 
If $vw\in E(G)$ and $v\in A$ and $w\in B$, then the box representing $v$ intersects the box representing $w$. 
For a non-edge $vw$ with $v\in A$ and $w\in B$, the box representing $v$ intersects the box representing $w$ if and only if $\sigma_i$ catches $vw$. 
Since for every non-edge $vw$, some $\sigma_i$ does not catch $vw$, 
the boxes representing $v$ and $w$ do not intersect. Thus
$\bx(H\langle A,B\rangle ) \leq 2t \leq 2+ 3 (k+1) \log n$. 

By \cref{sumboxicity},
\begin{align*}
\bx(H) \leq  & \bx(H\langle A\rangle ) + \bx(H\langle B\rangle ) + \bx( H\langle A,B \rangle ) \\
\leq 
& (k+2) \ceil{2e \log n} + 
\frac{3g}{k-6} + 
2+ 3 (k+1) \log n \\
\leq
& (9k+15)  \log (3002g)  + 
3 \sqrt{g\log g} \\
\leq
& 12 \sqrt{g\log g} + O( \sqrt{g/\log g} ). 
\end{align*}
Applying \cref{sumboxicity} again,  
\begin{align*}
\bx(G) & \leq \bx(G_1) + \bx(G_2) + \bx(G_3) \\
& \leq 42 + 48 \log\log (1000g) +  12 \sqrt{g\log g} + O( \sqrt{g/\log g} ) \\
& \leq 12 \sqrt{g\log g} + O( \sqrt{g/\log g} ).\qedhere
\end{align*}
\end{proof}

As noted in \cref{Introduction}, when combined with the lower bound proved by \citet{Esperet16}, this shows that the maximum possible boxicity of a graph with Euler genus $g$ is $\Theta(\sqrt{g\log g})$.

%%%%%%%%%%%%%%%%%%%%%%%%%%%%%%%%%%
\section{Layered Treewidth}
\label{layeredtreewidth}

A \emph{tree decomposition} of a graph $G$ is a set $(B_x:x\in V(T)$ of non-empty sets $B_x\subseteq V(G)$ (called \emph{bags}) indexed by the nodes of a tree $T$, such that for each vertex $v\in V(G)$, the set $\{x\in V(T):v\in B_x\}$ induces a non-empty (connected) subtree of $T$, and for each edge $vw\in E(G)$ there is a node $x\in V(T)$ such that $v,w\in B_x$. The \emph{width} of a tree decomposition $(B_x:x\in V(T))$ is $\max\{|B_x|-1: x\in V(T)\}$. The \emph{treewidth} of a graph $G$, denoted by $\tw(G)$, is the minimum width of a tree decomposition of $G$. Treewidth is a key parameter in algorithmic and structural graph theory (see \citep{Reed97,Bodlaender-TCS98,HW17} for surveys).
\citet{CS07} proved: 

\begin{thm}[\citep{CS07}] 
\label{tw}
For every graph $G$, 
  \begin{equation*}
\bx(G)\leq \tw(G)+2.
\end{equation*}
\end{thm}

A \emph{layering} of a graph $G$ is a partition $(V_1,V_2,\dots,V_n)$ of $V(G)$ such that for every edge $vw\in E(G)$, for some $i\in[n-1]$, both $v$ and $w$ are in $V_i\cup V_{i+1}$. For example, if $r$ is a vertex of a connected graph $G$ and $V_i:=\{v\in V(G):\dist(r,v)=i\}$ for $i\geq 0$,  then $(V_0,V_1,\dots)$ is a layering of $G$. The layered tree-width of a graph $G$ is the minimum integer $k$ such that there is a tree decomposition  $(B_x:x\in V(T))$ and a layering $(V_1,V_2,\dots,V_n)$  of $G$, such that $|B_x\cap V_i| \leq k$ for each node $x\in V(T)$ and for each layer $V_i$. Of course, $\ltw(G) \leq \tw(G)+1$ and often $\ltw(G)$ is much less than $\tw(G)$. For example, \citet{DMW17} proved that every planar graph has layered treewidth at most 3, whereas the $n\times n$ planar grid has treewidth $n$. Thus the following result provides a qualitative generalisation of \cref{tw}. 

\begin{thm}
\label{ltw}
For every graph $G$, 
  \begin{equation*}
\bx(G) \leq 6 \ltw(G) + 4.
\end{equation*}
\end{thm}

\begin{proof}
Consider a tree decomposition  $(B_x:x\in V(T))$ and a layering $(V_1,V_2,\dots,V_n)$  of $G$, such that $|B_x\cap V_i| \leq \ltw(G)$ for each node $x\in V(T)$ and for each layer $V_i$. Note that $(B_x\cap (V_i\cup V_{i+1}) :x\in V(T))$ is a tree-decomposition of $G[V_i\cup V_{i+1}]$ with bags of size at most $2\ltw(G)$. Thus $\tw(G[V_i\cup V_{i+1}]) \leq 2\ltw(G) - 1$. 
For $i\in\{0,1,2\}$, let 
  \begin{equation*}
G_i := \bigcup_{j\equiv i\, (\text{mod }{3})} \!\!\! G[V_j \cup V_{j+1} ] .
\end{equation*}
Each component of $G_i$ is contained in $V_j \cup V_{j+1}$ for some $j\equiv i \pmod{3}$. 
The treewith of a graph equals the maximum treewidth of its connected components. 
Thus $\tw(G_i)\leq 2\ltw(G) -1$, and $\bx(G_i) \leq 2\ltw(G) + 1$ by \cref{tw}. 
%Add each vertex in $\cup\{V_{j+2}: j \equiv i \pmod{3}\}$ to each dimension of the the representation of $G_i$ with interval $\mathbb{R}$. 
Use three disjoint sets of $2\ltw(G) + 1$ dimensions for each $G_i$, and add each vertex not in $G_i$ to the dimensions used by $G_i$ with interval $\mathbb{R}$. 
Finally, add one more dimension, where the interval for each vertex $v\in V_i$ is $[i,i+1]$. 
For adjacent vertices in $G$, the corresponding boxes intersect in every dimension. 
Consider non-adjacent vertices $v$ and $w$ in $G$. 
Say $v\in V_a$ and $w\in V_b$. If $|a-b|\geq 2$ then in the final dimension, the intervals for $v$ and $w$ are disjoint, as desired. 
If $|a-b|=1$, then $vw$ is a non-edge in some $G_i$, and thus the intervals for $v$ and $w$ are disjoint in some dimension corresponding to $G_i$. 
%Thus the boxes corresponding to $v$ and $w$ do not intersect in the dimensions corresponding to $G_i$.
Hence we have a $3(2\ltw(G) + 1)+1$-dimensional box representation of $G$. 
\end{proof}

The following two examples illustrate the generality of \cref{ltw}. A graph is \emph{$(g,k)$-planar} if it has a drawing in a surface of Euler genus at most $g$ with at most $k$ crossings per edge; see \citep{PachToth97,KLM17,Schaefer14} for example. \citet{DEW17} proved that every $(g,k)$-planar graph has layered treewidth at most  $(4g + 6)(k + 1)$. \cref{ltw} then implies that every $(g,k)$-planar graph has boxicity at most $6 (4g + 6)(k + 1) + 4$. Map graphs provide a second example. Start with a graph $G_0$ embedded in a surface of Euler genus $g$, with each face labelled a `nation' or a `lake', where each vertex of $G_0$ is incident with at most $d$ nations. Let $G$ be the graph whose vertices are the nations of $G_0$, where two vertices are adjacent in $G$ if the corresponding faces in $G_0$ share a vertex. Then $G$ is called a \emph{$(g,d)$-map graph}; see \citep{Chen07,CGP02} for example.  \citet{DEW17} proved that every $(g, d)$-map graph has layered treewidth at most $(2g + 3)(2d + 1)$. \cref{ltw} then implies that every $(g,d)$-map graph has boxicity at most $6 (2g + 3)(2d + 1)+ 4$. By definition, a graph is $(g,0)$-planar if and only if it has Euler genus at most $g$. Similarly, it is easily seen that a graph is a $(g,3)$-map graph if and only if it has Euler genus at most $g$ (see \citep{DEW17}). Thus these results provide qualitative generalisations of the fact that graphs with Euler genus $g$ have boxicity $O(g)$, as proved by \citet{EJ13}. As discussed above, \citet{Esperet16} improved this upper bound to $O(\sqrt{g}\log g)$ and \cref{genus} improves it further to $O(\sqrt{g\log g})$. On the other hand, \cref{ltw} is within a constant factor of optimal, since \citet{CS07} constructed a family of graphs $G$ with $\bx(G)\geq (1-o(1))\tw(G)$. See \citep{BDDEW18} for more examples of graph classes with bounded layered treewidth, for which \cref{ltw} is applicable. 

%%%%%%%%%%
\section{Open Problems}
\label{questions}

We conclude with a few open problems. 

\begin{itemize}
\item  What is the maximum boxicity of graphs with maximum degree 4?

\item What is the maximum boxicity of $k$-degenerate graphs with maximum degree $\Delta$?

\item What is the maximum boxicity of graphs with treewidth $k$? \citet{CS07} proved lower and upper bounds of $k-2\sqrt{k}$ and $k+2$ respectively. %\comment{I suspect the answer is $k+1$ --- something to look at in April perhaps? }

\item What is the maximum boxicity of graphs with no $K_t$ minor? 
%It follows from a result of \citet{HOQRS17} that such graphs are acyclically $O(t^2)$-colourable. \cref{acyclic} then shows that such graphs have boxicity $O(t^4)$. \citet{EJ13} previously observed a $O(t^4\log^2 t)$ upper bound. 
The best known upper bound is $O(t^2\log t)$ due to \citet{EW18}. 
A lower bound of $\Omega( t\sqrt{\log t} )$ follows from results of \citet{Esperet16}.

%\item \comment{Very speculative and rough. Include?} 
%Let $H_k$ be the graph obtained from $K_{2k}$ by deleting a perfect matching. 
%Let $K'_k$ be the 1-subdivision of $K_k$. 
%It is known that $\bx(H_k) = k$ (or thereabouts) and that $\bx(K'_k)\geq\log\log k$. 
%For a graph $G$, let $f(G)$ be the maximum integer $k$ such that $H_k$ or $K'_k$ is an induced subgraph of $G$. Thus $\bx(G) \geq \log\log f(G)$. Is there are a function $h$ such that $\bx(G) \leq h( f(G) )$ for every graph $G$? If not, what other graphs are needed in addition to $H_k$ and $K'_k$. Do finitely many graphs suffice?

%\item Whether $\bx(\Delta)\leq O(\Delta \log\Delta)$ remains a tantalising open problem. 

\end{itemize}

\subsection*{Acknowledgement} Thanks to Louis Esperet and Tom Trotter for useful comments. The proofs of \cref{bxdim,dimbx} are due to Tom Trotter.

%\bibliographystyle{myNatbibStyle}
%\bibliography{myBibliography}
%\bibliography{Boxicity}

\def\soft#1{\leavevmode\setbox0=\hbox{h}\dimen7=\ht0\advance \dimen7
  by-1ex\relax\if t#1\relax\rlap{\raise.6\dimen7
  \hbox{\kern.3ex\char'47}}#1\relax\else\if T#1\relax
  \rlap{\raise.5\dimen7\hbox{\kern1.3ex\char'47}}#1\relax \else\if
  d#1\relax\rlap{\raise.5\dimen7\hbox{\kern.9ex \char'47}}#1\relax\else\if
  D#1\relax\rlap{\raise.5\dimen7 \hbox{\kern1.4ex\char'47}}#1\relax\else\if
  l#1\relax \rlap{\raise.5\dimen7\hbox{\kern.4ex\char'47}}#1\relax \else\if
  L#1\relax\rlap{\raise.5\dimen7\hbox{\kern.7ex
  \char'47}}#1\relax\else\message{accent \string\soft \space #1 not
  defined!}#1\relax\fi\fi\fi\fi\fi\fi}

\end{document}